\documentclass[a4paper,american,reqno]{amsart}

\newif\ifPreprint \Preprinttrue
\newif\ifSubmission \Submissionfalse

\usepackage{babel}
\usepackage[utf8]{inputenc}
\usepackage[T1]{fontenc}
\usepackage{xspace}
\usepackage{todonotes}
\usepackage{csquotes}
\usepackage{microtype}
\usepackage{xcolor}
\usepackage[style = authoryear-comp,
            maxbibnames = 100,
            maxcitenames = 2,
            giveninits = true,
            uniquename = init,
            isbn = false,
            backend = bibtex]{biblatex}
\usepackage[colorlinks,
            citecolor=blue,
            urlcolor=blue,
            linkcolor=blue]{hyperref} 

\makeatletter
\patchcmd{\@settitle}{\uppercasenonmath\@title}{\scshape\large}{}{}
\patchcmd{\@setauthors}{\MakeUppercase}{\scshape\normalsize}{}{}
\makeatother

\vfuzz \hfuzz
\newcommand{\YB}{\todo[author=YB, color=green!50, size=\small]}

\newcommand{\field}{\mathbb}

\newcommand{\reals}{\field{R}}

\newcommand{\R}{\reals}

\newcommand{\st}{\text{s.t.}}

\makeatletter
\newcommand{\fcdot}{\,\cdot\,}
\newcommand{\fcarg}[1]{\def\fc@rg{#1}\ifx\fc@rg\empty\fcdot\else\fc@rg\fi}
\makeatother
\newcommand{\abs}[1]{\lvert\fcarg{#1}\rvert}
\newcommand{\Abs}[1]{\left\lvert#1\right\rvert}

\newcommand{\defset}[3][\defsep]{\set{#2#1#3}}
\newcommand{\Defset}[3][\defsep]{\Set{#2#1#3}}
\newcommand{\set}[1]{\{#1\}}
\newcommand{\Set}[1]{\left\{#1\right\}}

\DeclareMathOperator*{\argmin}{arg\,min}

\newcommand{\abbr}[1][abbrev]{#1.\xspace}

\newcommand{\eg}{\abbr[e.g]}

\newcommand{\ie}{\abbr[i.e]}

\newcommand{\Wlog}{\abbr[w.l.o.g]}
\newcommand{\wrt}{\abbr[w.r.t]}

\newcommand{\define}{\mathrel{{\mathop:}{=}}}

\ifPreprint
\newtheorem{theorem}{Theorem}

\newtheorem{remark}{Remark}
\newtheorem{lemma}{Lemma}

\newtheorem{definition}{Definition}

\fi
\newtheorem{observation}{Observation}
\newtheorem{result}{Result}

\newcommand{\myeps}{\varepsilon}

\newcommand{\rev}[1]{#1}

\newcommand{\revFirst}[1]{#1}
\newcommand{\revSec}[1]{#1}
\newcommand{\revThird}[1]{#1}

\bibliography{nearly-feasible-bilevel}

\begin{document}

\title[On Bilevel Problems with Continuous and Nonconvex Lower Levels]{
  On a Computationally Ill-Behaved Bilevel Problem\\
  with a Continuous and Nonconvex Lower Level}

\author[Y. Beck, D. Bienstock, M. Schmidt, J. Thürauf]{
  Yasmine Beck, Daniel Bienstock, Martin Schmidt, Johannes Thürauf}

\address[Y. Beck, M. Schmidt, J. Thürauf]{%
  Trier University,
  Department of Mathematics,
  Universitätsring 15,
  54296 Trier,
  Germany}

\email{yasmine.beck@uni-trier.de}
\email{martin.schmidt@uni-trier.de}
\email{johannes.thuerauf@uni-trier.de}

\address[D. Bienstock]{%
  Columbia University,
  Department of Industrial Engineering and Operations Research,
  500 W.\ 120th Street \#315,
  New York, NY 10027,
  United States}

\email{dano@columbia.edu}

\date{\today}

\begin{abstract}
  It is well known that bilevel optimization problems are hard to solve
both in theory and practice.
In this paper, we highlight a further computational difficulty
when it comes to solving bilevel problems with continuous but
nonconvex lower levels.
Even if the lower-level problem is solved to $\myeps$-feasibility
regarding its nonlinear constraints for an arbitrarily small but
positive $\myeps$,
the obtained bilevel solution as well as its objective value may
be arbitrarily far away from the actual bilevel solution and its
actual objective value.
This result even holds for bilevel problems
for which the nonconvex lower level is uniquely solvable, for which
the strict complementarity condition holds, \revFirst{for which the
  feasible set is convex, and for which Slater's constraint
qualification is satisfied for all feasible upper-level decisions}.
Since the consideration of $\myeps$-feasibility cannot be avoided
when solving nonconvex problems to global optimality, our result
shows that computational bilevel optimization with continuous and
nonconvex lower levels needs to be done with great care.
Finally, we illustrate that the nonlinearities in the lower level are
the key reason for the observed bad behavior by showing that
linear bilevel problems behave much better at least on the level
of feasible solutions.

\ifSubmission
\begin{center}
  \vspace*{1em}
  \emph{Communicated by Miguel F. Anjos}
\end{center}
\fi


\end{abstract}

\keywords{Bilevel optimization,
Nonconvex lower levels,
Approximate feasibility,
Global optimization%
%
%
}
\subjclass[2010]{90-XX, 
90C26, 
90C31
%
}

\maketitle

\section{Introduction}
\label{sec:introduction}

Bilevel optimization problems are known to be notoriously hard to
solve and this holds true both in theory and in practice.
In theory, bilevel problems are strongly NP-hard
even if all objective functions and constraints are linear and all
variables are continuous; see \textcite{Hansen-et-al:1992}.
This, of course, is also reflected in computational practice since
linear bilevel problems are inherently nonsmooth and nonconvex problems.
Moreover, the single-level reformulations used to solve linear bilevel
problems in practice are nonconvex, complementarity-constrained
problems.
Their linearization requires big-$M$ parameters that are
hard to obtain in general \parencite{Kleinert_et_al:2020} and that
often lead to numerically badly posed problems, which are hard to
tackle even for state-of-the-art commercial solvers.

In the last years, algorithmic research on bilevel optimization
focused on more and more complicated lower-level problems such as
mixed-integer linear models
\parencite{Fischetti-et-al:2017,Fischetti-et-al:2018b},
nonlinear but still convex models
\parencite{Kleinert_et_al:2021b},
or problems in the lower level with uncertain data
\parencite{Burtscheidt-Claus:2020,Buchheim_Henke:2022,Beck_et_al:2022}.
When it comes to the situation of a lower-level problem with
continuous nonlinearities there is not too much literature---in
particular in comparison to the case in which the lower-level problem
is convex; see, \eg, \textcite{Mitsos-et-al:2008,Mitsos:2010,%
  Kleniati-Adjiman:2011,Kleniati-Adjiman:2014a,Kleniati-Adjiman:2014b,%
  Kleniati-Adjiman:2015,Paulavicius-et-al:2016,Paulavivcius-et-al:2020,%
  Paulavicius-Adjiman:2020}.
Due to the brevity of this article, we do not go into the details of
the literature but refer to the seminal
\ifPreprint
book by~\textcite{Dempe:2002} as well as the recent survey
by~\textcite{Kleinert_et_al:2021c} for further discussions of the
relevant literature.
\fi
\ifSubmission
book~\textcite{Dempe:2002} as well as the recent
survey~\textcite{Kleinert_et_al:2021c} for further discussions of the
relevant literature.
\fi

There is one important difference when crossing the border from
(mixed-integer) convex to (mixed-integer) nonconvex lower-level
problems.
The lower-level problem can, in general, not be solved to global
optimality anymore in an exact sense in finite time since we need to
exploit  techniques such as spatial branching to tackle nonconvexities.
These techniques only lead to finite algorithms for prescribed and
strictly positive feasibility tolerances;
see, e.g.,~\textcite{Locatelli_Schoen:2013} for more detailed
discussions.
Note that this is in clear contrast to, \eg, linear optimization.
Here, the simplex method is an exact method in the sense that, if
applied using exact arithmetic, the method computes a global optimal
solution without any errors; see, \eg, \textcite{Applegate_et_al:2007}.
The same applies to simplex-based branch-and-bound methods for
solving mixed-integer linear optimization problems.
Algorithms as such are not available for continuous but nonconvex
problems.
This means that, just due to algorithmic necessities, we cannot expect
to get exact feasible solutions of the lower-level problem anymore
when doing computations for continuous but nonconvex lower-level
  problems.
Instead, we have to deal with $\myeps$-feasible solutions---at least
for the nonlinear constraints of the lower-level problem.

The aim of this paper is to present an exemplary bilevel
optimization problem with continuous variables and a nonconvex
lower-level problem, where the latter algorithmic aspect leads to the
following severe issue:

\begin{quote}
  \emph{Even if the feasibility tolerance for the lower level can be
    made  extremely small, the exact bilevel solution can be
    arbitrarily far away from the bilevel solution that one obtains
    for $\myeps$-feasibility in the lower level, which in
      particular can be superoptimal out of proportion to $\myeps$.
    The same is true for the optimal objective function values.}
\end{quote}

The main idea for the construction of this exemplary bilevel problem
is based on a constraint set presented first
\ifSubmission
in~\textcite{Bienstock_et_al:2021}.
\fi
\ifPreprint
by~\textcite{Bienstock_et_al:2021}.
\fi
We explicitly note here that this construction does not make use of
large constraint coefficient ranges (all coefficients are~$1$)
or arbitrarily large degrees of polynomials
(we only use quadratic or linear terms).
Moreover, when considered in an exact sense, (i) the example's lower-level
problem is uniquely solvable, (ii) strict complementarity holds,
(iii) its convex constraint set satisfies
Slater's constraint qualification for all feasible upper-level
decisions, (iv) the upper level does not contain coupling
constraints, and (v) the overall problem has a unique global solution
as well.
Thus, the bilevel program does not look like a badly-modeled problem
but is shown to behave very badly in a computational sense, i.e.,
if only $\myeps$-feasible points for the nonlinear constraints of the
lower-level problem can be considered.
We also show that the observed pathological behavior arises due
to the nonlinearities as we show that linear bilevel problems behave
much better at least on the level of feasible points.

The example is presented in Section~\ref{sec:problem-statement} and
discussed in an exact sense in Section~\ref{sec:exact}.
Afterward, the example is analyzed in Section~\ref{sec:inexact} for
the case of $\myeps$-feasibility of nonlinear constraints. Section
\ref{sec:finitelinlin} presents an analysis of the linear bilevel case.
Our final conclusions are drawn in Section~\ref{sec:conclusion}.


\section{Problem Statement}
\label{sec:problem-statement}

Let us consider the bilevel problem
\begin{subequations} \label{eq:UL-prob}
  \begin{align}
    \max_{x \in \R^2} \quad & F(x,y) = x_1 - 2y_{n+1} + y_{n+2} \label{eq:UL-obj} \\
    \st \quad & (x_1,x_2) \in [\underbar{$x$}_1,\bar{x}_1]
                \times [\underbar{$x$}_2,\bar{x}_2], \\
                 & y \in S(x),
  \end{align}
\end{subequations}
where~$\underbar{$x$},\, \bar{x} \in \R^2$
with~$1 \leq \underbar{$x$}_i < \bar{x}_i$, $i \in \set{1,2}$,
denote lower and upper bounds on the variables~$x$.
Here,~$S(x)$ is the set of optimal solutions of
the~$x$-parameterized~problem
\begin{subequations} \label{eq:LL-prob}
  \begin{align}
    \max_{y \in \R^{n+2}} \quad & f(x,y) = y_1 - y_n \left(x_1 + x_2
                                  - y_{n+1} - y_{n+2}\right) \\
    \st \quad & y_1 + y_n = \frac{1}{2}, \label{eq:eq-constr} \\
                                & y_i^2 \leq y_{i+1},
                                  \quad i \in \set{1,\ldots,n-1},
                                  \label{eq:quadr-constr} \\
                                & y_i \geq 0, \quad i \in \set{1,\ldots,n},
                                  \label{eq:non-neg} \\
                                & y_{n+1} \in [0,x_1], \label{eq:var-bounds-1} \\
                                & y_{n+2} \in [-x_2,x_2]. \label{eq:var-bounds-2}
  \end{align}
\end{subequations}
We refer to Problem~\eqref{eq:UL-prob} as the upper-level (or the leader's)
problem and to Problem~\eqref{eq:LL-prob} as the lower-level
(or the follower's) problem. Let us point out that the lower-level
constraints~\eqref{eq:eq-constr} and~\eqref{eq:quadr-constr}
together with~$y_1 \geq 0$ have already been considered
in \textcite{Bienstock_et_al:2021} in the context of approximately
feasible solutions for single-level optimization~problems.
Let us further emphasize that the number of variables and
constraints of the lower-level problem is linear in~$n$.

Problem~\eqref{eq:UL-prob} is a linear problem
in both the leader's and the follower's variables.
The only constraints that occur in this problem are variable
bounds for the leader's variables~$x$.
In particular, there are no upper-level constraints that explicitly
depend on the follower's variables~$y$, \ie,
there are no coupling constraints.

Moreover, the feasible set of the lower-level problem is bounded due to
the following. From~\eqref{eq:non-neg} and~\eqref{eq:eq-constr},
we obtain~\mbox{$0 \leq y_1 \leq 1/2$} as well as~\mbox{$0 \leq y_n \leq 1/2$}
for any feasible follower's decision~$y$.
Using Constraints~\eqref{eq:quadr-constr}, we further
obtain~\mbox{$0 \leq y_i \leq \revSec{1}$} for all~$i \in \set{1,\ldots,n}$.
Finally, we have~$0 \leq y_{n+1} \leq \bar{x}_1$ as well
as~$-\bar{x}_2 \leq y_{n+2} \leq \bar{x}_2$ by~\eqref{eq:var-bounds-1}
and~\eqref{eq:var-bounds-2} because the leader's variables~$x$ are bounded.
Since all finitely many lower-level constraints are continuous, the
feasible set of the follower's problem is compact.
In addition to the compactness, the feasible set of the lower-level
problem~\eqref{eq:LL-prob} is non-empty for every feasible leader's
decision~$(x_1,x_2) \in [\underbar{$x$}_1,\bar{x}_1]
\times [\underbar{$x$}_2,\bar{x}_2]$.
For instance, the point
\begin{equation*}
  y_i = \frac{i}{2^{2^n}},\, i \in \set{1,\ldots,n-1}, \quad
  y_n = \frac{1}{2} - \frac{1}{2^{2^n}}, \quad
  y_{n+1} = \frac{1}{2}, \quad \text{and} \quad
  y_{n+2} = 0
\end{equation*}
is strictly feasible \wrt\ the inequality
constraints~\eqref{eq:non-neg}, \eqref{eq:var-bounds-1}
as well as~\eqref{eq:var-bounds-2} and feasible
\wrt\ the equality constraint~\eqref{eq:eq-constr}.
Here, we exploit the assumption that~\mbox{$1 \leq \underbar{$x$}_1,
  \underbar{$x$}_2$} holds to obtain strict feasibility \wrt\ the
variable bounds in~\eqref{eq:var-bounds-1}
and~\eqref{eq:var-bounds-2}.
Moreover,~$y$ is also strictly feasible \wrt\ the inequality
constraints~\eqref{eq:quadr-constr} due to the following.
For all~$i \in \set{1,\ldots,n-2}$, we have
\begin{align*}
  y_{i+1} - y^2_i
  = \frac{i+1}{2^{2^n}} - \left(\frac{i}{2^{2^n}}\right)^2
  = \frac{2^{2^n}(i+1) - i^2}{\left( 2^{2^n} \right)^2 }
  = \frac{2^{2^n} + 2^{2^n}i \left( 1 - \frac{i}{2^{2^n}} \right)}{
  \left( 2^{2^n} \right)^2}
  > 0.
\end{align*}
Furthermore, we have
\begin{align*}
  y_n - y^2_{n-1}
  & = \frac{1}{2} - \frac{1}{2^{2^n}} -
    \left( \frac{n-1}{2^{2^n}} \right) ^2
    = \frac{\left( 2^{2^n} \right) ^2 - 2\cdot 2^{2^n} - 2(n-1)^2}{
    2 \cdot \left( 2^{2^n} \right)^2}
  \\
  & = \frac{2^{2^n} \left( 2^{2^n} - 2 - \frac{2(n-1)^2}{2^{2^n}}\right)}{
  2 \cdot \left( 2^{2^n} \right)^2}
  > 0.
\end{align*}
In particular, this means that the \revFirst{problem satisfies
  Slater's constraint qualification.}
Moreover, the gradient of the single equality
constraint~\eqref{eq:eq-constr} is not the null vector.
Hence, the Mangasarian--Fromovitz constraint qualification (MFCQ) is
also satisfied at every feasible decision of the follower.
Let us further point out that all lower-level constraints are linear except
for the quadratic but convex inequality constraints in~\eqref{eq:quadr-constr}.
Therefore, the feasible set of the lower-level problem~\eqref{eq:LL-prob}
is convex.
Nevertheless, the overall lower-level problem is nonconvex since
the follower's objective function contains bilinear terms.

Before we solve the bilevel problem~\eqref{eq:UL-prob} and~\eqref{eq:LL-prob}
in the following sections, let us brief\/ly summarize the nice properties
of the problem.
The upper-level problem is linear and does not contain coupling constraints.
The feasible set of the lower-level problem is convex and compact.
For every feasible leader's decision, the lower-level problem
further satisfies Slater's constraint qualification \revSec{and the MFCQ
is satisfied for every feasible follower's decision}.


\section{Exact Feasibility}
\label{sec:exact}

In this section, we determine the unique exact solution of the
bilevel problem~\eqref{eq:UL-prob} and~\eqref{eq:LL-prob}.
To this end, we start by solving the lower-level problem~\eqref{eq:LL-prob}
analytically for an arbitrary but fixed feasible leader's
decision~\mbox{$(x_1,x_2) \in [\underbar{$x$}_1,\bar{x}_1]
  \times [\underbar{$x$}_2,\bar{x}_2]$}.

First, we note that any feasible follower's decision~$y$ satisfies~$y_n > 0$.
The reasons are as follows. Let us \revFirst{contrarily} assume that~$y_n = 0$ holds.
Then, Constraint~\eqref{eq:eq-constr} yields~\mbox{$y_1 = 1/2$}.
\revFirst{
  From~$y_{n} = 0$ and~\eqref{eq:quadr-constr}, it follows that~$y_i = 0$ holds
  for all~\mbox{$i \in \set{1,\ldots,n}$}, which contradicts~$y_{1} = 1/2$.
  Consequently, $y_{n} > 0$ holds.}
For later reference, let us brief\/ly summarize the previous observation.

\begin{result}
  For every feasible leader's
  decision~$(x_1,x_2) \in [\underbar{$x$}_1,\bar{x}_1]
  \times [\underbar{$x$}_2,\bar{x}_2]$, a feasible follower's decision~$y$
  satisfies~$y_n > 0$.
\end{result}

The equality constraint~\eqref{eq:eq-constr} thus yields~$y_1 < 1/2$.
From~\eqref{eq:var-bounds-1} and~\eqref{eq:var-bounds-2},
we additionally obtain
\begin{equation*}
  y_n \left( x_1 + x_2 - y_{n+1} - y_{n+2} \right) \geq 0.
\end{equation*}
In particular, the latter term is minimized for~$(y_{n+1},y_{n+2}) = (x_1,x_2)$.
Therefore, the lower-level objective function value can be bounded
from above by
\begin{equation*}
  f(x,y) = y_1 - y_n \left(x_1 + x_2 - y_{n+1} - y_{n+2} \right)
  \leq y_1 < \frac{1}{2}.
\end{equation*}
It is thus evident that an optimal follower's decision~$y^*$
satisfies~$(y^*_{n+1},y^*_{n+2}) = (x_1,x_2)$.
\rev{Here, we can fix~$(y^*_{n+1},y^*_{n+2})$ since these variables
  are subject to simple bound constraints and, in particular,
  they are not coupled to the other variables of the follower.}
Hence, the follower's problem can be reduced to the convex problem
\begin{subequations} \label{eq:reduced-LL-prob}
  \begin{align}
    \max_y \quad & y_1 \\
    \st \quad & y_1 + y_n = \frac{1}{2}, \label{eq:eq-constr-red} \\
                 & y_i^2 \leq y_{i+1}, \quad i \in \set{1,\ldots,n-1},
                   \label{eq:quadr-constr-red} \\
                 & y_i \geq 0, \quad i \in \set{1,\ldots,n}.
  \end{align}
\end{subequations}
As shown above, \revFirst{Problem~\eqref{eq:reduced-LL-prob}
satisfies Slater's constraint qualification.}\footnote{It can also be
  shown that the linear independence constraint qualification is valid
  at all solutions of the follower's problem (for any given leader's
  decision~$x$).
  \revSec{For the latter, see Appendix~\ref{sec:appendix-licq}.}}
Again as shown above, the feasible set is compact. Therefore,
Problem~\eqref{eq:reduced-LL-prob} has an optimal solution~$y^*$.
Because of the equality constraint~\eqref{eq:eq-constr-red}, the lower-level
objective function value~$y^*_1$ is maximized by minimizing~$y^*_n$.
\revFirst{From} Constraints~\eqref{eq:quadr-constr-red} \revFirst{and
  the optimality of~$y^{*}$}, \revFirst{we obtain}
\begin{equation*}
  \rev{y^*_i = \left( y^*_1 \right) ^{2^{i-1}}}
  \quad \text{ for all } i \in \set{2,\ldots,n},
\end{equation*}
where~$y^*_1$ denotes the root of the function
\begin{equation} \label{eq:unique-sol}
  h: \left[ 0,\frac{1}{2} \right] \to \R,
  \quad
  z \mapsto z + \rev{z^{2^{n-1}}} - \frac{1}{2}.
\end{equation}
In particular, one can show that $y^*_1$ is the unique root
of~\eqref{eq:unique-sol}.
The function~$h$ is continuous and
strictly increasing on~$[0,1/2]$.
Moreover, we have~$h(0) < 0$ and~$h(1/2) > 0$.
Consequently, there is a unique point~$y^*_1 \in (0,1/2)$ such
that~\mbox{$h(y^*_1) = 0$} holds.
Furthermore, the follower's decision~$y^*$ is the unique solution of
Problem~\eqref{eq:reduced-LL-prob}. To see this, let us assume
that there is another feasible follower's decision~$\hat{y} \neq y^*$
for which the optimal objective function value~$y^*_1$ is obtained, \ie,
$\hat{y}_1 = y^*_1$. Then, there must be at least one quadratic inequality
constraint in~\eqref{eq:quadr-constr-red} that is not satisfied with equality
for~$\hat{y}$. Otherwise, we have~$y^* = \hat{y}$. However, if there is
slack in Constraints~\eqref{eq:quadr-constr-red}, we obtain~$\hat{y}_n > y^*_n$.
Then,~\eqref{eq:eq-constr-red} yields
\begin{equation*}
  y^*_1 = \frac{1}{2} - y^*_n > \frac{1}{2} - \hat{y}_n = \hat{y}_1,
\end{equation*}
which is a contradiction to the optimality of~$\hat{y}$.

\begin{result} \label{thm:singleton}
  For every feasible leader's
  decision~$(x_1,x_2) \in [\underbar{$x$}_1,\bar{x}_1]
  \times [\underbar{$x$}_2,\bar{x}_2]$, the set of optimal solutions of
  the lower-level problem~\eqref{eq:LL-prob} is a singleton.
\end{result}

In particular, Result~\ref{thm:singleton} means that there is no need
to distinguish between the optimistic and the pessimistic approach
to bilevel optimization; see, \eg,~\textcite{Dempe:2002}.
Thus, we can finally determine an optimal leader's decision for
the overall bilevel problem~\eqref{eq:UL-prob} and~\eqref{eq:LL-prob}.
As~$(y^*_{n+1},y^*_{n+2}) = (x_1,x_2)$ holds in the optimal follower's
decision~$y^*$, the leader actually solves the linear problem
\begin{equation*}
  \max_x \quad -x_1 + x_2
  \quad \st \quad (x_1,x_2) \in [\underbar{$x$}_1,\bar{x}_1]
  \times [\underbar{$x$}_2,\bar{x}_2].
\end{equation*}
The unique optimal solution is given by~$x^* = (\underbar{$x$}_1,\bar{x}_2)$.

\begin{result}
  The bilevel problem~\eqref{eq:UL-prob} and~\eqref{eq:LL-prob} has
  a unique~solution given by~\mbox{$x^* = (\underbar{$x$}_1,\bar{x}_2)$}
  with an optimal objective function value
  of~$F^* = -\underbar{$x$}_1 + \bar{x}_2$.
\end{result}

To sum up, the bilevel problem~\eqref{eq:UL-prob} and~\eqref{eq:LL-prob}
not only has nice properties such as a convex and bounded lower-level
feasible set as well as a lower-level problem that satisfies
Slater's constraint qualification, but also has a unique optimal solution.
\rev{Moreover, the strict complementarity condition holds for which we give
a proof in Appendix~\ref{sec:appendix}.}
Overall, the bilevel problem~\eqref{eq:UL-prob} and~\eqref{eq:LL-prob} is
thus well-behaved.


\section{$\myeps$-Feasibility}
\label{sec:inexact}

In what follows, we determine an optimal solution of the bilevel
problem~\eqref{eq:UL-prob} and~\eqref{eq:LL-prob} under the assumption
that we allow for small violations of the nonlinear lower-level constraints
according to the following notion\revFirst{, which is motivated by the
  necessary special treatment of nonlinear (and, in particular,
  nonconvex) constraints in global optimization as we discussed it in
  the introduction.}

\begin{definition}
  \revFirst{Let~$0 < \myeps \in \R$, $f: \R^n \rightarrow \R$,
  $g: \R^n \rightarrow \R^m$, and $h: \R^n \rightarrow \R^p$ be given.
  A point~\mbox{$x \in \R^n$} is called \emph{$\myeps$-feasible} for
  the problem~$\max_{x \in \R^n} \defset{f(x)}{g(x) \leq 0, h(x) = 0}$
  if \mbox{$g_i(x) \leq 0$} and $h_j(x) = 0$ holds for all $i \in
  \set{1,\ldots,m} \setminus N$ as well as for all $j \in
  \set{1,\ldots,p} \setminus M$ and if~$\max \set{\max \defset{g_i(x)}{i \in N}, \max
    \defset{\abs{h_j(x)}}{j \in M}} \leq \myeps$ holds, where~$N
  \subseteq \set{1,\ldots,m}$ and $M \subseteq \set{1,\ldots,p}$
  denote the index sets of all nonlinear inequality and equality
  constraints.}
\end{definition}

A follower's decision of the form
\begin{equation} \label{eq:eps-feas-y}
  y_i = 2^{{-2}^{i-1}},\, i \in \set{1,\ldots,n-1}, \
  y_n = 0, \
  y_{n+1} \in [0,x_1], \ \text{and} \
  y_{n+2} \in [-x_2,x_2]
\end{equation}
is~$\myeps$-feasible with $\myeps = 2^{-2^{n-1}}$ for every feasible
leader's decision~$(x_1,x_2) \in [\underbar{$x$}_1,\bar{x}_1]
\times [\underbar{$x$}_2,\bar{x}_2]$ due to the following.
The constraints
\begin{align*}
  y_1 + y_n & = \frac{1}{2}, \\
  y^2_i & \leq y_{i+1}, \quad i \in \set{1,\ldots,n-2}, \\
  y_i & \geq 0, \quad i \in \set{1,\ldots,n}, \\
  y_{n+1} & \in [0,x_1], \\
  y_{n+2} & \in [-x_2,x_2]
\end{align*}
are (exactly) satisfied, whereas only the constraint~$y^2_{n-1} \leq y_n$
is violated by \mbox{$\myeps = 2^{-2^{n-1}}$}.
Moreover, the lower-level objective function value is~$1/2$.

\begin{result}
  \rev{If~$\myeps \geq 2^{-2^{n-1}}$}, there is an~$\myeps$-feasible
  follower's decision~$y$ with~$y_n = 0$ for every feasible leader's
  decision~$(x_1,x_2) \in [\underbar{$x$}_1,\bar{x}_1]
  \times [\underbar{$x$}_2,\bar{x}_2]$.
\end{result}

It can easily be seen that by increasing~$n$, we can obtain arbitrarily
small values for $\myeps$.
In particular, there is no $\myeps$-feasible follower's decision
that yields a better objective function value than~$1/2$. The reasons are
as follows. Using the equality constraint~\eqref{eq:eq-constr}, the
lower-level objective function can be re-written as
\begin{equation*}
  f(x,y) = \frac{1}{2} - y_n - y_n \left(x_1 + x_2 - y_{n+1} - y_{n+2}\right).
\end{equation*}
For all~$\myeps$-feasible follower's decisions, we have
\begin{equation*}
  x_1 + x_2 - y_{n+1} - y_{n+2} \geq 0
\end{equation*}
because of the linear constraints~\eqref{eq:var-bounds-1}
and~\eqref{eq:var-bounds-2}. Consequently, a lower-level objective
function value larger than~$1/2$ could only be obtained if~$y_n < 0$. However,
this is not $\myeps$-feasible w.r.t.\ the variable
bounds~\eqref{eq:non-neg}.
In addition, a follower's decision of the form stated
in~\eqref{eq:eps-feas-y} is thus an $\myeps$-feasible solution
of the lower-level problem~\eqref{eq:LL-prob}.
Let us point out that, in contrast to the exact case,
the follower's variables~$y_{n+1}$ and~$y_{n+2}$
do not affect the lower-level objective function value in this setting
and can thus be chosen arbitrarily.
Therefore, the set of $\myeps$-feasible follower's solutions is not a
singleton anymore.

\begin{result} \label{thm:no-singleton}
  \rev{If~$\myeps \geq 2^{-2^{n-1}}$}, the set of $\myeps$-feasible
  follower's solutions is not a singleton for every feasible leader's
  decision~$(x_1,x_2) \in [\underbar{$x$}_1,\bar{x}_1]
  \times [\underbar{$x$}_2,\bar{x}_2]$.
\end{result}

Due to Result~\ref{thm:no-singleton}, we need to distinguish
between optimistic and pessimistic solutions.
Following the optimistic approach, the follower chooses~$y_{n+1} = 0$
as well as~$y_{n+2} = x_2$ such as to favor the leader \wrt\ the
leader's objective function value.
Therefore, the leader actually solves the linear problem
\begin{equation*}
  \max_x \quad x_1 + x_2
  \quad \st \quad
  (x_1,x_2) \in [\underbar{$x$}_1,\bar{x}_1] \times
  [\underbar{$x$}_2,\bar{x}_2].
\end{equation*}
The optimistic optimal leader's decision is thus given
by~$x^* = (\bar{x}_1,\bar{x}_2)$.
In the pessimistic case, the follower chooses~$y_{n+1} = x_1$ as well
as~$y_{n+2} = -x_2$ such as to adversely affect the leader's decision.
In this setting, the leader solves the linear problem
\begin{align*}
  \max_x \quad -x_1 - x_2
  \quad \st \quad (x_1,x_2) \in [\underbar{$x$}_1,\bar{x}_1]
  \times [\underbar{$x$}_2,\bar{x}_2].
\end{align*}
Hence, the pessimistic optimal leader's decision is given
by~$x^* = (\underbar{$x$}_1,\underbar{$x$}_2)$.
To sum up, let us state the main observations of this section.

\begin{result}
  \label{res:final}
  Let~$\myeps \geq 2^{{-2}^{n-1}}$ and suppose that we allow
  for~$\myeps$-feasible follower's solutions.
  Then, the optimistic optimal solution
  of the bilevel problem~\eqref{eq:UL-prob} and~\eqref{eq:LL-prob}
  is given by~$x_{\text{o}}^* = (\bar{x}_1,\bar{x}_2)$ with an optimal objective
  function value of~$F^*_{\text{o}} = \bar{x}_1 + \bar{x}_2$.
  The pessimistic optimal solution is given
  by~$x_{\text{p}}^* = (\underbar{$x$}_1,\underbar{$x$}_2)$ with an optimal
  objective function value of~$F^*_{\text{p}} = -\underbar{$x$}_1-\underbar{$x$}_2$.
\end{result}

We now finally compare the results of the exact bilevel solution with
the results for the optimistic and pessimistic setting for the case of
only $\myeps$-feasibility of the lower level.
In the optimistic setting, the distance between the solutions is
$\bar{x}_1 - \underbar{$x$}_1$ and the difference between the
corresponding objective function values is $\bar{x}_1 +
\underbar{$x$}_1$.
Two aspects are remarkable.
First, by enlarging the feasible interval for the variable~$x_1$, we
get an arbitrarily large error and, second, this error is
independent of $\myeps$, \ie, this arbitrarily large error occurs
independent of how accurate one solves the lower-level problem.

For the pessimistic setting, the distance between the solution is
$\bar{x}_2 - \underbar{$x$}_2$ and the difference between the
objective function values is $\bar{x}_2 + \underbar{$x$}_2$.
Hence, we obtain the same qualitative behavior but now in dependence
of the variable~$x_2$ instead of $x_1$.

In summary, we obtain the following two main observations.
First, we can be arbitrarily far away from the overall exact bilevel
solution.
Second, we also obtain arbitrarily
large errors regarding the optimal objective function value
of the leader.
The latter is very much in contrast to the situation in single-level
optimization for which sensitivity results are available;
see, \eg, Proposition~4.2.2 in \textcite{Bertsekas:2016}.
This is particularly the case for linear optimization problems,
where standard sensitivity analysis results (see, \eg, Theorem~5.5
in~\textcite{Chvatal:1983}) apply as well and state
that a small change in the right-hand side of the problem's
constraints can only lead to a small change in the optimal objective
function value.

Lastly, let us comment on that only very moderate values of $n$ are
required to get the wrong solution.
Taking the inequality for $\myeps$ from Result~\ref{res:final}, it is
easy to see that for a given tolerance~$\myeps$, the parameter~$n$
needs to satisfy $n \geq \log_2(\log_2(1/\myeps^2))$ so that
\revFirst{numerically computed solutions do not coincide with
  the exact solution} for the given $\myeps$.
For instance, a tolerance of $\myeps = 10^{-8}$ already leads to a
wrong result for $n = 6$.
This particularly means that the considered bilevel problem
is moderate in size \wrt \ the number of constraints and
variables. For~$n=6$, we only have~\rev{16} constraints
and~8 variables on the lower~level.
We further note that the used constraint coefficients are all~1 and
that the coefficients are independent from~$n$ and the given
tolerance~$\myeps$.

A \textsf{Python} code for the example considered in this paper is
publicly available at
\url{https://github.com/m-schmidt-math-opt/ill-behaved-bilevel-example}
and can be used to verify the discussed results.


\section{Analysis of the $\myeps$-Feasible Linear Case}
\label{sec:finitelinlin}

In this section, we analyze the linear bilevel case, \ie, we study the problem
\begin{subequations} \label{eq:linearlinear}
  \begin{align}
    \min_{x, y} \quad
    & c_x^\top x + c_y^\top y
      \label{eq:linlinobj} \\
    \st \quad
    & A x \geq a, \\
    & y \in \argmin_{\bar{y}}\Defset{d^\top\bar{y}}{Cx + D\bar{y} \geq b}
      \label{eq:linear-lower-level}
  \end{align}
\end{subequations}
with~$c_x \in \R^{n_x}$, $c_y,\, d \in \R^{n_y}$, $A \in \R^{m \times
  n_x}$, $a \in \R^m$, $C \in \R^{\ell \times n_x}$, $D \in \R^{\ell
  \times n_y}$, and~$b \in \R^{\ell}$.
We assume that the set~$\defset{(x,y) \in \R^{n_x} \times
  \R^{n_y}}{Ax \geq a,\, Cx + Dy \geq b}$ is non-empty and compact and
that for every feasible upper-level decision~$x$, there exists a
feasible lower-level decision~$y$.
This implies that the lower-level problem is bounded for every
feasible upper-level decision and that the
dual problem of the lower level is feasible.
We also assume that the set $\defset{x \in \R^{n_x}}{Ax \geq a}$
is bounded.
Moreover, we consider the setting in which the
underlying linear algebra and linear optimization routines are of
finite precision only.

When finite-precision procedures are used, an algorithm that solves
Problem~\eqref{eq:linearlinear} will output a pair~$(\hat{x}, \hat{y})$ that
may be slightly infeasible.
The concern, should that happen, is that the solution being output can be
\emph{superoptimal} to a degree that is not proportional to its infeasibility.
As discussed in the previous section, such an outcome can be observed for
general, \ie, nonlinear, bilevel problems.
In this section, however, we show that linear bilevel problems behave
better in some sense.
To this end, we assume that our underlying solver can ensure the
following properties:
\begin{itemize}
\item $A\hat{x} \geq  a - \myeps e_m$ and $C\hat{x} + D\hat{y} \geq b - \myeps e_{\ell}$,
\item $d^\top \hat{y} \geq \min \defset{d^\top y}{C\hat{x} + Dy \geq b} - \myeps$.
\end{itemize}
Here and in what follows, $0 < \myeps < 1$ is a given tolerance, $e_k
\in \R^k$ is the vector of all ones, and $(\hat{x},\hat{y})$ is used to
denote a nearly feasible solution of the bilevel
problem~\eqref{eq:linearlinear}.

Prior to our analysis, we present a general result that will be used
below.
This result can be read from Theorem~3.38 (Page~112) of
\textcite{Conforti2014}.
It can also be obtained from Corollary~3.2b (Page~20) of
\textcite{Schrijver86} or from Theorem~10.2 (Page~121) of
\textcite{Schrijver86}.
We will use the term \textit{size} to refer to the (bit) encoding
length of a matrix, vector, or formulation, as appropriate.

\begin{definition} \label{def:primalbasic}
  Let~$P = \defset{x \in \R^n}{Hx = h, \, x \geq 0}$
  \revThird{with~$H \in \R^{m \times n}$ and~\mbox{$h \in \R^m$}}.
  Given $z \in \R^n$,
  we say that $z$ is \emph{basic} if $H z = h$ and, defining
  $B = B(z) = \defset{j}{z_j \neq 0}$,
  the submatrix $H_B$ of
  $H$ corresponding to the columns in~$B$ has rank~$|B|$.
  Furthermore, if in addition $z \ge 0$, we say that $z$ is
  \emph{basic feasible}.
\end{definition}

\begin{remark}
  Let $P$ be as in Definition \ref{def:primalbasic}. The
  extreme points of $P$ are precisely the vectors~$z$ that are basic
  feasible.
\end{remark}

\begin{theorem}
  \label{thm:sizesours}
  Let $P = \defset{x \in \R^n}{Hx = h, \, x \geq 0}$
  \revThird{with~$H \in \R^{m \times n}$ and~\mbox{$h \in \R^m$}}.
  There is a constant~\revThird{$\kappa(H) > 0$} of size polynomial
  in the size (of the bit-encoding) of $H$ such that, for any
  basic vector~$v$, we have
  \begin{equation*}
    \|v\|_{\infty} \leq \kappa(H) \|h\|_{\infty}.
  \end{equation*}
\end{theorem}
\begin{proof}
  Let $v$ be basic and set $J = \defset{j}{v_j \neq  0}$.
  Since $v$ is basic, there is a subset
  of rows $I$ of $H$ with $|I| = |J|$ such that the following holds:
  \begin{itemize}
  \item [(i)] The submatrix $H_{I,J}$ of $H$ indexed by rows $I$ and
    columns $J$ is invertible.
  \item[(ii)] As a consequence, it holds $v_J = H_{I,J}^{-1} h_I$, where $v_J$ is the
    subvector of $v$ indexed by $J$ and $h_I$ is the subvector of $h$
    indexed by $I$.
  \end{itemize}
  Using submultiplicativity of the norm, we get
  $$\|v\|_{\infty} \le \| H_{I,J}^{-1}\|_{\infty} \| h_I \|_{\infty}.$$
  The result now follows by defining
  $ \kappa(H)$ to be the maximum over all $\| B^{-1} \|_{\infty}$
  for~$B$ being an invertible submatrix of $H$.
  \ifSubmission
  \qed
  \fi
\end{proof}

In Theorem~\ref{thm:sizesours}, we use what is usually termed the
\textit{standard} representation of a polyhedron. Similar statements
can be derived using other representations of polyhedra, e.g., $\defset{x
\in \R^n}{Hx \leq h}$, via well-known reformulations.

\subsection{Linear Optimization with Errors}
\label{sub:singlelevel}

We start with some simple observations for classic, \ie,
single-level, linear problems of the form
\begin{equation} \label{eq:single-lev-LP}
  v^* \define \min_{x \in \R^{n_x}} \Defset{v^\top x}{Mx \geq f}
\end{equation}
with~$v \in \R^{n_x}$, $0 \neq M \in \R^{m \times n_x}$, and~$f \in \R^m$.
Throughout this section, we assume that the feasible region for
problem~\eqref{eq:single-lev-LP} is non-empty and bounded.
Moreover, we denote the corresponding dual problem by
\begin{equation*}
  \max_{z \in \reals^{m}} \Defset{f^{\top} z}{M^{\top}z = v, \ z \geq 0}.
\end{equation*}
Next, we will derive estimates involving near-feasible and near-optimal
points for Problem~\eqref{eq:single-lev-LP}.
\begin{lemma}
  \label{thm:sensitivity-lemma}
  Suppose that there is a point~$\hat{x} \in \R^{n_x}$ that is nearly feasible for
  Problem~\eqref{eq:single-lev-LP}, \ie, $M \hat{x} \geq f - \myeps e_{m}$.
  Then, the following holds.
  \begin{enumerate}
  \item[(a)] It holds
    \begin{equation*}
      v^\top \hat{x} \geq v^* - \myeps \kappa(M)\|v\|_{\infty},
    \end{equation*}
    where~$\kappa(M) > 0$ is a constant of polynomial size.
  \item[(b)] There exists $x^*$ feasible for
    Problem~\eqref{eq:single-lev-LP} such that
    \begin{equation*}
      \| x^* - \hat x\|_{\infty} \leq  \myeps \kappa_1(M)
    \end{equation*}
    holds for a certain constant $\kappa_1(M) > 0$ of polynomial size.
  \end{enumerate}
\end{lemma}
\begin{proof}
  \begin{enumerate}
  \item[(a)] Let~$z^*$ be an optimal solution of the dual problem
    of~\eqref{eq:single-lev-LP}.
    Then,
    \begin{equation*}
      v^\top \hat{x}
      = (z^*)^\top M \hat{x} \geq (z^*)^\top (f - \myeps e_m)
      = v^* - \myeps \|  z^* \|_1
    \end{equation*}
    holds. In particular,
    this equation applies to any dual optimal~$z^*$. Since~$z \geq 0$ is a constraint
    of the dual problem, the dual feasible region is a pointed polyhedron, and, \Wlog,
    $z^*$ is an extreme point. The result now follows from Theorem
    \ref{thm:sizesours}.
  \item[(b)] Consider the linear optimization problem
    \begin{subequations}
      \label{eq:primdist0}
      \begin{align}
        \min_{x, \delta} \quad
        & \delta
          \label{eq:inftnorm}\\
        \st \quad
        & \|x - \hat{x}\|_{\infty} \leq \delta,
          \label{eq:primdist0:inf-norm}\\
        & M x \geq f. \label{ineq:primal-feas0}
      \end{align}
    \end{subequations}
    That this is indeed a linear program follows by reformulating
    \eqref{eq:primdist0:inf-norm} as
    \begin{equation}
      \label{eq:reform}
      x_j - \delta \, \le \, \hat{x}_j,
      \quad
      -x_j - \delta \, \le \, -\hat{x}_j,
      \quad
      1 \le j \le n_x.
    \end{equation}
    Clearly, the resulting problem is both
    feasible and bounded since \eqref{eq:single-lev-LP} is.
    Moreover, $(\hat x, 0)$ satisfies the constraints of
    this problem with additive error of at most~$\myeps$.
    We can therefore apply (a) to this problem to obtain $x^*$
    being feasible for Problem~\eqref{eq:single-lev-LP} and such that
    $$ \| x^* - \hat x\|_\infty \leq \myeps\kappa_1(M)$$
    holds, where $\kappa_1(M)$ is the $\kappa$-constant (of polynomial
    size in $M$) that applies to the matrix for
    Constraints~\eqref{eq:reform} and \eqref{ineq:primal-feas0}.
    \ifSubmission
    \qed
    \fi
    \ifPreprint
    \qedhere
    \fi
  \end{enumerate}
\end{proof}

Let us emphasize that the result in Lemma~\ref{thm:sensitivity-lemma}
applies for any~$\myeps > 0$, no matter how large.
In particular, it is not required that~$\myeps$ is ``sufficiently
small''.

\begin{lemma} \label{thm:stability-lemma}
  Suppose that there is a nearly primal-dual feasible and nearly
  primal-dual optimal pair~$(\hat{x}, \hat{z}) \in \R^{n_x} \times \R^m$
  for Problem~\eqref{eq:single-lev-LP}, \ie,~$(\hat{x}, \hat{z})$
  satisfies
  \begin{itemize}
  \item [(i)] $M \hat{x} \geq f - \myeps e_{m}$,
  \item [(ii)] $ \| M^\top \hat{z} - v \|_{\infty} \leq \myeps$, $\hat{z}
    \geq -\myeps e_m$, and
  \item [(iii)] $v^\top \hat{x} - f^\top \hat{z} \leq \myeps$.
  \end{itemize}
  Then, there exists an optimal solution~$x^*$ for
  Problem~\eqref{eq:single-lev-LP} such that
  \begin{equation*}
    \| x^* - \hat{x} \|_{\infty} \leq \myeps \kappa_3(M, v) \max\{1,\| f \|_{\infty}\}
  \end{equation*}
  holds for a certain constant~$\kappa_3(M,v) > 0$, \rev{whose
    size is polynomial in the size of the input data~$M$
    and~$v$}.
 \end{lemma}
  \begin{proof}
    First, we note that Condition~(ii) simply states that $\hat z$ is
    feasible for the dual of~\eqref{eq:single-lev-LP} up to an error
    of~$\myeps$.  We can thus apply Part~(a) of
    Lemma~\ref{thm:sensitivity-lemma} to obtain
    \begin{align}
      \label{eq:dualsuper}
      f^\top \hat{z} \leq v^* + \myeps \kappa_2(M) \|f\|_{\infty},
    \end{align}
    where $\kappa_2(M)$ is the $\kappa$-constant for the dual
    of~\eqref{eq:single-lev-LP}, which is of polynomial size in $M$.
    Together with (iii), this implies
    \begin{align}
      \label{eq:dualsuper2}
      v^{\top} \hat x \leq v^* + \myeps (1 + \kappa_2(M) \|f\|_{\infty}).
    \end{align}
    Next, we consider the polyhedron given by
    \begin{subequations} \label{eq:primdist}
      \begin{align}
        M x & \geq f, \label{ineq:primal-feas} \\
        v^\top x & \leq v^*, \label{ineq:optimal-value-func}
      \end{align}
    \end{subequations}
    which is feasible and bounded.
    By (i) and~\eqref{eq:dualsuper2}, $\hat x$ satisfies these
    inequalities with feasibility error of at
    most $\myeps (1 + \kappa_2(M) \|f\|_{\infty})$.
    By applying Part~(b) of Lemma~\ref{thm:sensitivity-lemma}, we
    obtain that there is a feasible point $x^*$ for
    \eqref{eq:primdist}, i.e., an optimal solution~$x^*$
    for~\eqref{eq:single-lev-LP}, such that
    \begin{equation*}
      \| x^* - \hat{x} \|_\infty \leq \myeps \kappa_1(M, v)(1 +
      \kappa_2(M)\|f\|_{\infty})
    \end{equation*}
    holds.
    \ifSubmission
    \qed
    \fi
  \end{proof}

\subsection{Application to Linear Bilevel Problems}

We now return to the bilevel setup as stated in~\eqref{eq:linearlinear}.
To this end, note that for a given upper-level decision~$x \in
\R^{n_x}$, the dual of the lower-level
problem~\eqref{eq:linear-lower-level} reads
\begin{subequations} \label{eq:lowerdual}
  \begin{align}
    \max_{z} \quad & (b - Cx)^\top z  \\
    \st \quad & D^\top z = d, \\
                   & z \geq 0.
  \end{align}
\end{subequations}

\begin{lemma}
  \label{thm:bilevel-stability-lemma}
  Let $x \in \R^{n_x}$, $\hat{y} \in \R^{n_y}$, and $\hat{z} \in \R^{\ell}$
  be such that
  \begin{itemize}
  \item[(i)] $Ax \geq a$,
  \item[(ii)] $D\hat{y} \geq b - C\rev{x} - \myeps e_{\ell}$,
  \item[(iii)] $\| D^\top \hat{z} - d\|_{\infty} \leq \myeps$,
    $\hat{z} \geq -\myeps e_\ell$, and
  \item[(iv)] $\rev{d^\top \hat{y} - (b - Cx)^\top \hat{z}} \leq \myeps$.
  \end{itemize}
  Then, there exists an optimal solution~$y^*$ for the~\rev{$x$-parameterized}
  lower-level problem~\eqref{eq:linear-lower-level} such that
  \begin{equation}
    \label{eq:wearedone}
    \| y^* - \hat{y}\|_\infty \leq \myeps
    \kappa_4(A,C,D,a,b,d)
  \end{equation}
  holds for a constant~$\kappa_4(A,C,D,a,b,d) > 0$, whose
  size is polynomial in the size of the input data~$A$, $C$, $D$,
  $a$, $b$, and~$d$.
\end{lemma}
\begin{proof}
  By assumption, for a given $x$, the lower-level problem is feasible and
  bounded. We can thus apply Lemma~\ref{thm:stability-lemma} since
  Conditions (ii)--(iv) correspond to Conditions~\mbox{(i)--(iii)} of
  Lemma~\ref{thm:stability-lemma}.
  Thus, there exists an optimal point~$y^*$ for the lower-level
  problem such that
  \begin{equation*}
    \| y^* - \hat{y}\|_\infty
    \leq
    \myeps \kappa_3(D, d) \max\{1,\| b - Cx \|_{\infty} \}
  \end{equation*}
  holds.
  Using the triangle inequality and the submultiplicativity of the
  norm, we obtain
  \begin{equation*}
    \| b - Cx \|_{\infty} \, \leq \, \| b \|_{\infty} + \| C
    \|_{\infty} \|x\|_{\infty}.
  \end{equation*}
  Since the feasible region for the upper-level problem is bounded,
  $\| x \|_{\infty}$ is upper bounded by the $\infty$-norm of some
  extreme point. We can apply Theorem \ref{thm:sizesours} to obtain
  \begin{equation*}
    \| x \|_{\infty} \leq \kappa'(A) \| a \|_{\infty},
  \end{equation*}
  where $\kappa'(A)$ is the $\kappa$-constant (of polynomial size
  in~$A$) for the system $Ax \geq a$.
  The proof is now concluded by appropriately defining
  $\kappa_4(A,C,D,a,b,d)$.
  \ifSubmission
  \qed
  \fi
\end{proof}

Now, we consider the entire bilevel problem~\eqref{eq:linearlinear}
and recall a basic definition from linear optimization.

\begin{definition}
  Let $z \in \R^\ell$ satisfy $D^{\top} z = d$ and
  define~\mbox{$B = B(z) = \defset{j}{z_j \neq 0}$}.
  We say that~$z$ is \emph{dual basic} if the submatrix~$D^{\top}_B$ of
  $D^{\top}$ corresponding to the columns in~$B$ has rank~$|B|$.
\end{definition}

Note that $z$ is dual basic and feasible (i.e., $z \geq 0$) if and
only if $z$ is an extreme point of the dual polyhedron to any
lower-level problem.

\begin{theorem}
  \label{thm:main-thm}
  Let~$\hat{x} \in \R^{n_x}$, $\hat{y} \in \R^{n_y}$, and $\hat{z} \in \R^{\ell}$
  be such that
  \begin{itemize}
  \item[(i)] $A\hat{x} \geq a - \myeps e_m$,
  \item[(ii)] $D\hat{y} \geq b - C\hat{x} - \myeps e_{\ell}$,
  \item[(iii)] $\| D^\top \hat{z} - d\|_{\infty} \leq \myeps$,
    $\hat{z} \geq -\myeps e_\ell$,
  \item[(iv)] $d^\top \hat{y} - (b - C\hat{x})^\top \hat{z} \leq
    \myeps$,
  \item[(v)] $\| \hat{z} - \tilde z\|_\infty \leq \myeps$ for some
    dual basic  $\tilde z$.
  \end{itemize}
  Then, there exists a pair~$(x^*,y^*)$ that is feasible for the
  bilevel problem~\eqref{eq:linearlinear}
  such that
  \begin{align*}
    \| (x^*,y^*)^\top - (\hat{x},\hat{y})^\top \|_\infty
    & \leq \myeps \revThird{\kappa_5(A,C,D,a,b,d)},
    \\
    |c_x^\top x^* + c_y^\top y^* - (c_x^\top \hat{x} +
    c_y^\top \hat{y}) |
    & \leq \myeps \revThird{\kappa_6(A,C,D,a,b,c,d)}
  \end{align*}
  hold for certain constants~$\kappa_5(A,C,D,a,b,d)$ and
  \revThird{$\kappa_6(A,C,D,a,b,c,d) > 0$}, whose sizes
  are polynomial in the size of the input data.
\end{theorem}
\begin{proof}
  By Assumptions~\revThird{(i) and~(ii)},
  the pair~$(\hat{x},\hat{y}) \in \R^{n_x} \times \R^{n_y}$
  is nearly feasible for \revThird{the upper- and the lower-level
  problem of~\eqref{eq:linearlinear}}.
  Applying Part~(b) of
  Lemma~\ref{thm:sensitivity-lemma} to~$(\hat{x},\hat{y})$
  and the system
  \begin{equation} \label{eq:fullsystem}
    \begin{bmatrix}
      A & 0\\
      C & D
    \end{bmatrix}
    \begin{pmatrix}
      x\\
      y
    \end{pmatrix}
    \geq
    \begin{pmatrix}
      a\\
      b
    \end{pmatrix}
  \end{equation}
  yields~$(x^*,y')$ with~$Ax^* \geq a$, $Dy' \geq b - Cx^*$,
  and
  \begin{equation} \label{eq:dist-bilevel-feas-pair}
    \| (x^*,y')^\top - (\hat{x},\hat{y})^\top\|_\infty
    \leq \myeps \revThird{\kappa_7(A,C,D)},
  \end{equation}
  where~\revThird{$\kappa_7(A,C,D)$} is the $\kappa$-constant for
  System~\eqref{eq:fullsystem}.
  Next, we use~(iv) and obtain
  \begin{equation}
    \label{eq:longone}
    \begin{split}
      & \ d^\top y' - (b - Cx^*)^\top \hat{z}
      \\
      \leq & \ \myeps + d^\top (y' - \hat{y})
      - \revThird{\left(C(\hat{x} - x^*)\right) ^\top \hat{z}}
    \\
      \leq & \ \myeps + \Abs{d^\top (\hat{y} - y')}
      + \revThird{\Abs{\left(C(\hat{x} - x^*)\right)^\top \hat{z}}}.
    \end{split}
  \end{equation}
  Note that, if $u,\, v \in \R^n$, then $|u^\top v| \le \sum_{j
    = 1}^n |u_j v_j| \le  n \|u\|_\infty \|v\|_\infty$ holds.
  Moreover, if \mbox{$Q \in \R^{m \times n}$} and $u \in \R^n$, then, for $1
  \le i \le m$,  $| (Q u)_i| \le (\sum_{j = 1}^n |q_{ij}|) \| u
  \|_\infty \, \le \, \| Q \|_\infty \| u \|_\infty$ by definition of
  the infinity-norm of a matrix.
  Hence, using \eqref{eq:longone} we obtain
  \begin{equation}
    \label{eq:norms}
    d^\top y' - (b - Cx^*)^\top \hat z
    \leq  \ \myeps + n_y \|d\|_\infty \|\hat{y}-y'\|_\infty
      + \ell \|C\|_\infty \|\hat{x}-x^*\|_\infty \|\hat{z}\|_\infty.
  \end{equation}
  Further, since $\tilde z$ is dual basic, $\| \tilde z \|_{\infty}$ is upper
  bounded by \revThird{$\kappa(D)\|d\|_\infty$} due to Theorem~\ref{thm:sizesours}.
  Thus, using (v) yields $ \| \hat{z}  \|_\infty \leq \myeps +
  \revThird{\kappa(D)\|d\|_\infty}$.
  These facts, together with \eqref{eq:dist-bilevel-feas-pair} and
  \eqref{eq:norms}, yield
  \begin{equation*}
    d^\top y' - (b - Cx^*)^\top \hat z
    \leq
    \myeps \revThird{\kappa_8(A,C,D,d)},
  \end{equation*}
  with \revThird{$\kappa_8(A,C,D,d) \geq 1$} being appropriately
  defined and of polynomial size.
  To sum up, $x^*$, $y'$, and $\hat z$ satisfy
  \begin{itemize}
  \item[(a)] $A x^* \geq a$,
  \item[(b)] $D y' \geq b - C x^*$,
  \item[(c)] $\| D^\top \hat z - d\|_{\infty} \leq \myeps$,
    $\hat{z} \geq -\myeps e_\ell$,
  \item[(d)] $d^\top y' - (b - C x^*)^\top \hat z \leq
    \myeps \revThird{\kappa_8(A,C,D,d)}$.
  \end{itemize}
  Thus, by Lemma~\ref{thm:bilevel-stability-lemma} applied to the
  error \revThird{$\myeps \kappa_8(A,C,D,d) \geq \myeps$},
  there exists an optimal
  solution~$y^*$ for the~$x^*$-parameterized lower-level
  problem~\eqref{eq:linear-lower-level} such that
  \begin{equation}
    \label{eq:wearedone2}
    \| y^* - y'\|_\infty \leq  \myeps \revThird{\kappa_8(A,C,D,d)}
    \kappa_4(A,C,D,a,b,d)
  \end{equation}
  holds.
  Using this inequality and \eqref{eq:dist-bilevel-feas-pair}
  concludes the proof.
  \ifSubmission
  \qed
  \fi
\end{proof}

\begin{remark}
  \label{rem:basic}
  Assumption~(v) in Theorem~\ref{thm:main-thm} states that the
  distance between a nearly feasible and nearly optimal solution for
  the dual of the lower-level problem and a basic solution for the dual
  is small,
  which is a reasonable assumption in our setting.
\end{remark}

To summarize the statement of the theorem, the distance to feasibility
and the superoptimality of a nearly feasible pair~$(\hat{x}, \hat{y})$
for the bilevel problem~\eqref{eq:linearlinear} is linear in~$\myeps$
with coefficients~$\kappa$ that have polynomial size in the input data.
This type of guarantee with polynomial sized coefficients is simply
unavailable in the nonlinear case as we have seen in the previous
sections.


\section{Conclusion}
\label{sec:conclusion}

In this paper, we consider an exemplary bilevel problem with
continuous variables and a nonconvex lower-level problem and
illustrate that numerically obtained solutions can be arbitrarily far
away from an exact solution.
\revFirst{The discrepancy between exact and numerically computed
  solutions is based on the fact that we cannot exactly satisfy}
all constraints of the nonconvex lower level when
using global optimization techniques such as spatial branching.
The considered problem itself is well-posed in the sense that we
do not use large constraint coefficient ranges or
high-degree polynomials.
Moreover, we show that the constraint set of the lower-level problem
is convex, compact, and that it satisfies Slater's constraint
qualification.
In an exact sense, we prove that the lower-level problem as well as
the overall bilevel problem possess unique solutions.
\revSec{It is further established that LICQ holds in every follower's
  solution for every feasible leader's decision.}
While working computationally, however, we can only expect to
obtain~$\myeps$-feasible solutions of the nonconvex lower-level problem.
Furthermore, the set of~$\myeps$-feasible follower solutions is
not a singleton anymore.
Thus, we determine both an optimal solution for the optimistic and
the pessimistic variant of the bilevel problem.
By doing so, we establish that not only the
obtained~$\myeps$-feasible bilevel solutions can be arbitrarily
far away from the overall exact bilevel solution
but that there can also be an arbitrarily large error in the
objective function value of the leader.

We also show that the pathological behavior observed for nonlinear
lower-level problems seems to be due to the nonlinearities by showing
that linear bilevel problems behave better at least on the level of
feasible points.
\revFirst{As an important question for future research, it is still
  open if one can prove that the bad behavior can also not appear
  for more general problems than linear ones, such as convex
  problems, in the lower level.}

Finally, our results show that computational bilevel optimization with
continuous but nonconvex lower levels needs to be done with great care
and that ex-post checks may be needed to avoid considering arbitrarily
bad points as ``solutions'' of the given bilevel problem.


\ifPreprint
\section*{Acknowledgements}
\fi

The second author is grateful for the funding received based
on an ARPA-E GO competition award.
Moreover, the third author thanks the Deutsche Forschungsgemeinschaft for their
support within projects A05 and B08 in the
Sonderforschungsbereich/Transregio 154 ``Mathematical
  Modelling, Simulation and Optimization using the Example of Gas
  Networks''.


\printbibliography

@book{Conforti2014,
  doi = {10.1007/978-3-319-11008-0},
  year = {2014},
  publisher = {Springer International Publishing},
  author = {Michele Conforti and G{\'{e}}rard Cornu{\'{e}}jols and Giacomo Zambelli},
  title = {Integer Programming}
}

@Book{Schrijver86,
 author = {Alexander Schrijver},
 title = {Theory of Linear and Integer Programming},
 publisher = {John Wiley \& Sons, Chichester},
 year = {1986},
}

@inproceedings{Bienstock_et_al:2021,
  author =	 {Daniel Bienstock and Alberto Del Pia and Robert Hildebrand},
  editor =	 {Singh, M. and Williamson, D. P.},
  publisher =	 {Springer, Cham.},
  booktitle =	 {Integer Programming and Combinatorial Optimization},
  series =	 {IPCO 2021},
  volume =	 {12707},
  title =	 {Complexity, Exactness, and Rationality in Polynomial
                  Optimization},
  year =	 {2021},
  pages =	 {58--72},
  doi =		 {10.1007/978-3-030-73879-2_5},
}

@book{Dempe:2002,
  author    = {Dempe, Stephan},
  publisher = {Springer},
  year      = {2002},
  doi       = {10.1007/b101970},
  title     = {Foundations of Bilevel Programming},
}

@article{Hansen-et-al:1992,
  author       = {Hansen, Pierre and Jaumard, Brigitte and Savard, Gilles},
  publisher    = {SIAM},
  year         = {1992},
  doi          = {10.1137/0913069},
  journal      = {{SIAM} Journal on Scientific and Statistical Computing},
  number       = {5},
  pages        = {1194--1217},
  title        = {New branch-and-bound rules for linear bilevel programming},
  volume       = {13},
}

@article{Fischetti-et-al:2017,
  author       = {Fischetti, Matteo and Ljubić, Ivana and Monaci, Michele and Sinnl, Markus},
  publisher    = {INFORMS},
  year         = {2017},
  doi          = {10.1287/opre.2017.1650},
  journal      = {Operations Research},
  number       = {6},
  pages        = {1615--1637},
  title        = {A New General-Purpose Algorithm for Mixed-Integer Bilevel Linear Programs},
  volume       = {65},
}

@article{Fischetti-et-al:2018b,
  author       = {Fischetti, Matteo and Ljubić, Ivana and Monaci, Michele and Sinnl, Markus},
  year         = {2018},
  doi          = {10.1007/s10107-017-1189-5},
  journal      = {Mathematical Programming},
  number       = {1-2},
  pages        = {77--103},
  title        = {On the use of intersection cuts for bilevel optimization},
  volume       = {172},
}

@article{Kleinert_et_al:2021b,
  title    = {Outer Approximation for Global Optimization of Mixed-Integer Quadratic Bilevel Problems},
  author   = {Thomas Kleinert and Veronika Grimm and Martin Schmidt},
  year     = 2021,
  journal  = {Mathematical Programming (Series B)},
  doi      = {10.1007/s10107-020-01601-2},
  url-opto = {http://www.optimization-online.org/DB_HTML/2019/12/7534.html},
  url-opus = {https://opus4.kobv.de/opus4-trr154/frontdoor/index/index/docId/302},
}

@inbook{Burtscheidt-Claus:2020,
  author    = {Burtscheidt, Johanna and Claus, Matthias},
  editor    = {Dempe, Stephan and Zemkoho, Alain},
  publisher = {Springer International Publishing},
  booktitle = {Bilevel Optimization: Advances and Next Challenges},
  year      = {2020},
  doi       = {10.1007/978-3-030-52119-6_17},
  pages     = {485--511},
  title     = {Bilevel Linear Optimization Under Uncertainty},
}

@article{Buchheim_Henke:2022,
  title =	 {The robust bilevel continuous knapsack problem with
                  uncertain coefficients in the follower’s objective},
  author =	 {Christoph Buchheim and Dorothee Henke},
  year =	 {2022},
  journal =	 {Journal of Global Optimization},
  doi =		 {10.1007/s10898-021-01117-9},
}

@article{Mitsos-et-al:2008,
  author       = {Mitsos, Alexander and Lemonidis, Panayiotis and Barton, Paul I},
  publisher    = {Springer},
  year         = {2008},
  doi          = {10.1007/s10898-007-9260-z},
  journal      = {Journal of Global Optimization},
  number       = {4},
  pages        = {475--513},
  title        = {Global solution of bilevel programs with a nonconvex inner program},
  volume       = {42},
}

@article{Mitsos:2010,
  author       = {Mitsos, Alexander},
  publisher    = {Springer},
  year         = {2010},
  doi          = {10.1007/s10898-009-9479-y},
  journal      = {Journal of Global Optimization},
  number       = {4},
  pages        = {557--582},
  title        = {Global solution of nonlinear mixed-integer bilevel programs},
  volume       = {47},
}

@incollection{Kleniati-Adjiman:2011,
  author    = {Kleniati, Polyxeni-M. and Adjiman, Claire S.},
  editor    = {Pistikopoulos, E.N. and Georgiadis, M.C. and Kokossis, A.C.},
  publisher = {Elsevier},
  booktitle = {21st European Symposium on Computer Aided Process Engineering},
  year      = {2011},
  doi       = {10.1016/B978-0-444-53711-9.50121-8},
  issn      = {1570-7946},
  pages     = {602--606},
  series    = {Computer Aided Chemical Engineering},
  title     = {Branch-and-Sandwich: An Algorithm for Optimistic Bi-Level Programming Problems},
  volume    = {29},
}

@article{Kleniati-Adjiman:2014a,
  author       = {Kleniati, Polyxeni-M. and Adjiman, Claire S.},
  publisher    = {Springer},
  year         = {2014},
  doi          = {10.1007/s10898-013-0121-7},
  journal      = {Journal of Global Optimization},
  number       = {3},
  pages        = {425--458},
  title        = {Branch-and-Sandwich: a deterministic global optimization algorithm for optimistic bilevel programming problems. Part I: Theoretical development},
  volume       = {60},
}

@Article{Kleniati-Adjiman:2014b,
  author       = {Kleniati, Polyxeni-M. and Adjiman, Claire S.},
  title	       = {Branch-and-Sandwich: a deterministic global
                   optimization algorithm for optimistic bilevel
                   programming problems. Part II: Convergence analysis
                   and numerical results},
  journal      = {Journal of Global Optimization},
  year	       = 2014,
  volume       = 60,
  number       = 3,
  pages	       = {459-481},
  doi	       = {10.1007/s10898-013-0120-8},
  publisher    = {Springer}
}

@article{Kleniati-Adjiman:2015,
  author       = {Kleniati, Polyxeni-M. and Adjiman, Claire S.},
  year         = {2015},
  doi          = {10.1016/j.compchemeng.2014.06.004},
  journal      = {Computers \& Chemical Engineering},
  pages        = {373--386},
  title        = {A generalization of the Branch-and-Sandwich algorithm: From continuous to mixed-integer nonlinear bilevel problems},
  volume       = {72},
}

@incollection{Paulavicius-et-al:2016,
  author    = {Paulavičius, Remigijus and Kleniati, Polyxeni-M. and Adjiman, Claire S.},
  editor    = {Kravanja, Zdravko and Bogataj, Miloš},
  publisher = {Elsevier},
  booktitle = {26th European Symposium on Computer Aided Process Engineering},
  year      = {2016},
  doi       = {10.1016/B978-0-444-63428-3.50334-9},
  issn      = {1570-7946},
  pages     = {1977--1982},
  series    = {Computer Aided Chemical Engineering},
  title     = {Global optimization of nonconvex bilevel problems: implementation and computational study of the Branch-and-Sandwich algorithm},
  volume    = {38},
}

@article{Paulavivcius-et-al:2020,
  author       = {Paulavi{č}ius, Remigijus and Gao, J and Kleniati, Polyxeni-M and Adjiman, CS},
  publisher    = {Elsevier},
  year         = {2020},
  doi          = {10.1016/j.compchemeng.2019.106609},
  journal      = {Computers \& Chemical Engineering},
  pages        = {106609},
  title        = {BASBL: Branch-And-Sandwich BiLevel solver. Implementation and computational study with the BASBLib test set},
}

@article{Paulavicius-Adjiman:2020,
  author       = {Paulavi{č}ius, Remigijus and Adjiman, Claire S},
  publisher    = {Springer},
  year         = {2020},
  doi          = {10.1007/s10898-020-00874-3},
  journal      = {Journal of Global Optimization},
  pages        = {1--29},
  title        = {New bounding schemes and algorithmic options for the Branch-and-Sandwich algorithm},
}

@article{Kleinert_et_al:2021c,
  title    = {A Survey on Mixed-Integer Programming Techniques in Bilevel Optimization},
  author   = {Thomas Kleinert and Martine Labbé and Ivana Ljubi\'c and Martin Schmidt},
  year     = 2021,
  journal  = {EURO Journal on Computational Optimization},
  doi      = {10.1016/j.ejco.2021.100007},
  url-opto = {http://www.optimization-online.org/DB_HTML/2021/01/8187.html},
  url-opus = {https://opus4.kobv.de/opus4-trr154/frontdoor/index/index/docId/361},
}

@book{Locatelli_Schoen:2013,
author = {Locatelli, Marco and Schoen, Fabio},
title = {Global Optimization},
publisher = {Society for Industrial and Applied Mathematics},
year = {2013},
doi = {10.1137/1.9781611972672},
address = {Philadelphia, PA},
}

@article{Kleinert_et_al:2020,
  author     = {Thomas Kleinert and Martine Labbé and Fränk Plein and Martin Schmidt},
  title      = {There's No Free Lunch: On the Hardness of Choosing a Correct Big-M in Bilevel Optimization},
  journal    = {Operations Research},
  year       = {2020},
  doi        = {10.1287/opre.2019.1944},
  pages      = {1716--1721},
  volume     = {68},
  number     = {6},
}

@book{Bertsekas:2016,
  title =	 {Nonlinear Programming},
  author =	 {Bertsekas, Dimitri P.},
  year =	 {2016},
  publisher =	 {Athena scientific Belmont},
  isbn =	 {978-1-886529-05-2},
}

@book{Chvatal:1983,
  title = {Linear Programming},
  author = {Chvátal, Vasek},
  series = {A Series of books in the mathematical sciences},
  publisher = {Freeman},
  address = {New York (N. Y.)},
  isbn = {0-7167-1195-8},
  year = {1983},
}

@article{Applegate_et_al:2007,
  title   = {Exact solutions to linear programming problems},
  journal = {Operations Research Letters},
  volume  = {35},
  number  = {6},
  pages   = {693--699},
  year    = {2007},
  doi     = {10.1016/j.orl.2006.12.010},
  author  = {David L. Applegate and William Cook and Sanjeeb Dash and Daniel G. Espinoza},
}

@article{Beck_et_al:2022,
  title    = {A Survey on Bilevel Optimization Under Uncertainty},
  author   = {Yasmine Beck and Ivana Ljubi\'c and Martin Schmidt},
  year     = {2023},
  journal  = {European Journal on Operational Research},
  doi = {10.1016/j.ejor.2023.01.008},
}

\appendix
\section{Proof of the Linear Independence Constraint Qualification}
\label{sec:appendix-licq}

Let~$(x_1,x_2) \in [\underbar{$x$}_1,\bar{x}_1] \times [\underbar{$x$}_2,\bar{x}_2]$
with~$1 \leq \underbar{$x$}_i < \bar{x}_i$, $i \in \set{1,2}$, be arbitrary but fixed.
Further, let~$y^*$ be the exact optimal solution of the follower for the given
leader's decision~$x$.
As shown in Section~\ref{sec:exact}, a follower's solution~$y^*$ satisfies~$y^*_i > 0$
for all~$i \in \set{1,\ldots,n+2}$.
This means that the non-negativity constraints~\eqref{eq:non-neg} as well as
the lower bound constraints in~\eqref{eq:var-bounds-1} and~\eqref{eq:var-bounds-2}
are inactive in an optimal follower's decision.
Conversely, all quadratic constraints~\eqref{eq:quadr-constr}
as well as the upper bound constraints in~\eqref{eq:var-bounds-1}
and~\eqref{eq:var-bounds-2} are active.
Hence, the Jacobian matrix of the single equality constraint and the
active inequality constraints in an optimal decision of the follower is given by
\begin{equation*}
  \begin{bmatrix}
    1 & & & & & 1 & & \\
    2y^*_1 & -1 \\
    & 2y^*_2 & -1 \\
    & & \ddots & \ddots \\
    & & & 2y^*_{n-2} & -1 \\
    & & & & 2y^*_{n-1} & -1 \\
    & & & & & & 1 \\
    & & & & & & & 1 \\
  \end{bmatrix}.
\end{equation*}
All matrix entries that are left blank here correspond to zeros.
It is easy to verify that the Jacobian matrix has full rank, \ie,
the linear independence constraint qualification holds.


\section{Proof of the Strict Complementarity Condition}
\label{sec:appendix}

Let~$(x_1,x_2) \in [\underbar{$x$}_1,\bar{x}_1] \times [\underbar{$x$}_2,\bar{x}_2]$
with~$1 \leq \underbar{$x$}_i < \bar{x}_i$, $i \in \set{1,2}$, be arbitrary but fixed.
For the~$x$-parameterized lower-level problem~\eqref{eq:LL-prob}, the Lagrangian function
reads
\begin{align*}
  \mathcal{L}(y, \alpha, \beta, \gamma, \delta^\pm, \pi)
  = & -y_1 + y_n \left(x_1 + x_2 - y_{n+1} - y_{n+2} \right)\\
    &  - \sum_{i=1}^{n-1} \alpha_i \left( y_{i+1} - y^2_i \right)
      - \sum_{i=1}^{n+1} \beta_iy_i\\
    & - \gamma \left( x_1 - y_{n+1} \right)
      - \delta^- \left( y_{n+2} + x_2 \right)\\
    & - \delta^+ \left( x_2 - y_{n+2} \right)
      - \pi \left( y_1 + y_n - \frac{1}{2} \right)
\end{align*}
with the \revSec{Lagrange} multipliers~$\alpha \in \R^{n-1}_{\geq 0}$,
$\beta \in \R^{n+1}_{\geq 0}$, $\gamma,\, \delta^\pm \in \R_{\geq 0}$,
and~$\pi \in \R$.
The KKT complementarity conditions of Problem~\eqref{eq:LL-prob} are
given by
\begin{subequations}
  \begin{align}
    \alpha_i \left( y_{i+1} - y^2_i \right) & = 0, \quad i \in \Set{1,\ldots,n-1},
    \label{eq:complementarity-alpha}\\
    \beta_i y_i & = 0, \quad i \in \Set{1,\ldots,n+1},
    \label{eq:complementarity-beta}\\
    \gamma \left( x_1 - y_{n+1} \right) & = 0,
    \label{eq:complementarity-gamma}\\
    \delta^- \left( y_{n+2} + x_2 \right) & = 0,
    \label{eq:complementarity-delta-}\\
    \delta^+ \left( x_2 - y_{n+2} \right) & = 0.
                                            \label{eq:complementarity-delta+}
  \end{align}
\end{subequations}
Let~$y^*$ be the exact optimal solution of the follower for the given
leader's decision~$x$.
Further, let~$\alpha^*,\,\beta^*,\, \gamma^*,\, (\delta^\pm)^*$,
and~$\pi^*$ be the corresponding \revSec{Lagrange} multipliers so that
$(y^*,\alpha^*,\beta^*,\gamma^*,(\delta^\pm)^*,\pi^*)$ is a KKT point
of the lower-level problem.
As shown in Appendix~\ref{sec:appendix-licq}, the linear independence
constraint qualification is valid at all solutions of the follower’s
problem (for any given leader’s decision $x$). Hence, the \revSec{Lagrange}
multipliers~$\alpha^*,\,\beta^*,\, \gamma^*,\, (\delta^\pm)^*$,
and~$\pi^*$ are uniquely determined.
Also, as observed in Section~\ref{sec:exact}, the follower's optimal
decision~$y^*$ satisfies~$(y^*_{n+1},y^*_{n+2}) = (x_1,x_2)$ and
\begin{equation*}
  y^*_i = \left( y^*_1 \right) ^{2^{i-1}}
  \quad \text{ for all } i \in \set{2,\ldots,n},
\end{equation*}
where~$y^*_1$ is the unique root of the function $h$ as given
in~\eqref{eq:unique-sol}.
We now show that the strict complementarity condition is satisfied.

\begin{observation} \label{obs:beta-delta-}
  The point~$(y^*,\alpha^*,\beta^*,\gamma^*,(\delta^\pm)^*,\pi^*)$
  satisfies~$\beta^*_i = 0$ for all~$i \in \set{1,\ldots,n+1}$
  as well as~$(\delta^-)^* = 0$.
\end{observation}
\begin{proof}
  For all~$i \in \set{1,\ldots,n}$, we have~$y^*_i > 0$.
  By~\eqref{eq:complementarity-beta}, we thus obtain~$\beta^*_i = 0$
  for all~$i \in \set{1,\ldots,n}$.
  From~\eqref{eq:complementarity-beta}, we further obtain~\mbox{$\beta^*_{n+1} = 0$}
  since~$y^*_{n+1} = x_1 \geq \underbar{$x$}_1 \geq 1 > 0$ holds.
  Finally, due to~$y^*_{n+2} + x_2 = 2x_2 \geq 2\underbar{$x$}_2 \geq 2 > 0$,
  \eqref{eq:complementarity-delta-} yields~$(\delta^-)^* = 0$.
\end{proof}

\begin{observation}
  The point~$(y^*,\alpha^*,\beta^*,\gamma^*,(\delta^\pm)^*,\pi^*)$
  satisfies~$\gamma^*,\, (\delta^+)^* > 0$.
\end{observation}
\begin{proof}
  Since $(y^*,\alpha^*,\beta^*,\gamma^*,(\delta^\pm)^*,\pi^*)$ is a
  KKT point,
  \begin{align*}
    & \nabla_{y_{n+1}} \mathcal{L} (y^*,\alpha^*,\beta^*,\gamma^*,(\delta^\pm)^*,\pi^*)
      = -y^*_n - \beta^*_{n+1} + \gamma^* = 0,
    \\
    & \nabla_{y_{n+2}} \mathcal{L} (y^*,\alpha^*,\beta^*,\gamma^*,(\delta^\pm)^*,\pi^*)
      = -y^*_n - (\delta^-)^* + (\delta^+)^* = 0
  \end{align*}
  are satisfied. By Observation~\ref{obs:beta-delta-} and~$y^*_n > 0$,
  we obtain~$\gamma^*,\, (\delta^+)^* > 0$.
\end{proof}

\begin{observation} \label{obs:alpha-1-pi}
  The point~$(y^*,\alpha^*,\beta^*,\gamma^*,(\delta^\pm)^*,\pi^*)$
  satisfies~$2\alpha^*_1y^*_1 = 1 + \pi^*$.
\end{observation}
\begin{proof}
  Since $(y^*,\alpha^*,\beta^*,\gamma^*,(\delta^\pm)^*,\pi^*)$ is a
  KKT point,
  \begin{equation*}
    \nabla_{y_1} \mathcal{L} (y^*,\alpha^*,\beta^*,\gamma^*,(\delta^\pm)^*,\pi^*)
    = -1 + 2\alpha^*_1 y^*_1 - \beta^*_1 - \pi^* = 0
  \end{equation*}
  is satisfied. Since, by Observation~\ref{obs:beta-delta-}, we have~$\beta^*_1 = 0$,
  we obtain~$2\alpha^*_1y^*_1 = 1 + \pi^*$.
\end{proof}

\begin{observation} \label{obs:alpha-i-alpha-i-1}
  The point~$(y^*,\alpha^*,\beta^*,\gamma^*,(\delta^\pm)^*,\pi^*)$
  satisfies~$2\alpha^*_iy^*_i = \alpha^*_{i-1}$ for all~$i \in \set{2,\ldots,n-1}$.
\end{observation}
\begin{proof}
  Since $(y^*,\alpha^*,\beta^*,\gamma^*,(\delta^\pm)^*,\pi^*)$ is a
  KKT point,
  \begin{equation*}
    \nabla_{y_i} \mathcal{L} (y^*,\alpha^*,\beta^*,\gamma^*,(\delta^\pm)^*,\pi^*)
    = -\alpha^*_{i-1} + 2\alpha^*_iy^*_i - \beta^*_i = 0
  \end{equation*}
  is satisfied for all~$i \in \set{2,\ldots,n-1}$.
  By Observation~\ref{obs:beta-delta-}, we obtain~$2\alpha^*_iy^*_i =
  \alpha^*_{i-1}$ for all~$i \in \set{2,\ldots,n-1}$.
\end{proof}

\begin{observation} \label{obs:alpha-n-1-pi}
  The point~$(y^*,\alpha^*,\beta^*,\gamma^*,(\delta^\pm)^*,\pi^*)$
  satisfies~$\pi^* = - \alpha^*_{n-1}$.
\end{observation}
\begin{proof}
  Since $(y^*,\alpha^*,\beta^*,\gamma^*,(\delta^\pm)^*,\pi^*)$ is a
  KKT point,
  \begin{equation*}
    \nabla_{y_n} \mathcal{L} (y^*,\alpha^*,\beta^*,\gamma^*,(\delta^\pm)^*,\pi^*)
    = x_1 + x_2 - y^*_{n+1} - y^*_{n+2} - \alpha^*_{n-1} - \beta^*_n - \pi^* = 0
  \end{equation*}
  is satisfied. By Observation~\ref{obs:beta-delta-}, we have~$\beta^*_n = 0$.
  Moreover, $(y^*_{n+1}, y^*_{n+2}) = (x_1,x_2)$ holds.
  Hence, we obtain~$\pi^* = - \alpha^*_{n-1}$.
\end{proof}

\begin{observation}
  The point~$(y^*,\alpha^*,\beta^*,\gamma^*,(\delta^\pm)^*,\pi^*)$
  satisfies~$\alpha^*_i > 0$ for all~$i \in \set{1,\ldots,n-1}$.
\end{observation}
\begin{proof}
  We show this by contradiction.
  Suppose that~$\alpha^*_1 = 0$ holds.
  By Observation~\ref{obs:alpha-i-alpha-i-1} and~$y^*_i > 0$ for
  all~$i \in \set{1,\ldots,n}$, we obtain~$\alpha^*_i = 0$
  for all~$i \in \set{1,\ldots,n-1}$.
  Then, Observation~\ref{obs:alpha-n-1-pi}
  yields~$\pi^* = -\alpha^*_{n-1} = 0$.
  By Observation~\ref{obs:alpha-1-pi}, however, we
  obtain~\mbox{$0 = 2\alpha^*_1y^*_1 = 1 + \pi^*$, \ie, $\pi^* = -1$}
  which is a contradiction to~$\pi^* = 0$.
  Consequently, $\alpha^*_1 > 0$ needs to holds and, thus, we
  have~$\alpha^*_i > 0$ for all~$i \in \set{1,\ldots,n-1}$
  by Observation~\ref{obs:alpha-i-alpha-i-1}.
\end{proof}

To sum up, we have
\begin{align*}
  & y^*_{i+1} = (y^*_i)^2 \quad \text{and} \quad \alpha^*_i > 0
  \quad \text{for all } i \in \Set{1,\ldots,n-1},\\
  & y^*_i > 0 \quad \text{and} \quad \beta^*_i = 0
    \quad \text{for all } i \in \Set{1,\ldots,n+1},\\
  & y^*_{n+1} = x_1 \quad \text{and} \quad \gamma^* > 0,\\
  & y^*_{n+2} > - x_2 \quad \text{and} \quad (\delta^-)^* = 0,\\
  & y^*_{n+2} = x_2 \quad \text{and} \quad (\delta^+)^* > 0,
\end{align*}
\ie, strict complementarity is satisfied.


\end{document}